\documentclass{tran-l}

\newtheorem{theorem}{Theorem}[section]
\newtheorem{lemma}[theorem]{Lemma}

\theoremstyle{definition}

\newtheorem{example}[theorem]{Example}

\newtheorem{corollary}[theorem]{Corollary}

\theoremstyle{remark}

\numberwithin{equation}{section}

\DeclareMathOperator*{\glb}{glb}
\DeclareMathOperator*{\spec}{spec}
\DeclareMathOperator*{\chop}{chop}
\DeclareMathOperator*{\rem}{rem}

\begin{document}

\title{Zero sets of univariate polynomials}


\author{Robert S. Lubarsky}
\address{Florida Atlantic University}
\curraddr{}
\email{Robert.Lubarsky@comcast.net}
\thanks{}

\author{Fred Richman}
\address{Florida Atlantic University}
\curraddr{}
\email{richman@fau.edu}
\thanks{}

\subjclass[2000]{Primary 03F65, 13A99}

\date{}

\dedicatory{}

\begin{abstract}
Let $L$ be the zero set of a nonconstant monic polynomial with complex
coefficients. In the context of constructive mathematics without countable
choice, it may not be possible to construct an element of $L$. In this paper
we introduce a notion of distance from a point to a subset, more general
than the usual one, that allows us to measure distances to subsets like $L$.
To verify the correctness of this notion, we show that the zero set of a
polynomial cannot be empty---a weak fundamental theorem of algebra. We also
show that the zero sets of two polynomials are a positive distance from each
other if and only if the polynomials are comaximal. Finally, the zero set of
a polynomial is used to construct a separable Riesz space, in which every
element is normable, that has no Riesz homomorphism into the real numbers.
\end{abstract}

\maketitle

\section{Quasidistance}

Let $T$ be a set of real numbers. A \textbf{lower bound} for $T$ is a real
number $b$ such that $b\leq t$ for all $t$ in $T$. By the \textbf{infimum}
of $T$ we mean, as usual, a lower bound $\inf T$ for $T$ such that if $\inf
T<r$, then there exists $t\in T$ such that $t<r$. The number $\sup T$ is
defined dually. By the \textbf{greatest lower bound} of $T$ we mean the
largest element, $\glb T$, of the set $B$ of lower bounds of $T$. As
the distinction between $\inf T$ and $\glb T$ in constructive
mathematics is central to this section, we pause to illustrate this
difference.

\begin{example}
\label{Example 1}Let $P$ be an arbitrary proposition and let 
\begin{equation*}
L=\left\{ 0:P\text{ or }\lnot P\right\} \cup \left\{ 1\right\} .
\end{equation*}%
Then the set $B$ of lower bounds of $L$ is $\left( -\infty ,0\right] $. Thus 
$\glb L=0$, but if $\inf L$ exists, then $P$ or $\lnot P$.
\end{example}

\begin{proof}
Elements of $\left( -\infty ,0\right] $ are certainly lower bounds of $L$.
Conversely, if $b>0$, then $b$ cannot be a lower bound of $L$, because if $0$
were not in $L$, then $P$ would have to be false, so $\lnot P$ would be
true, so $0$ would be in $L$. So elements of $B$ cannot be (strictly)
positive, whence they are in $\left( -\infty ,0\right] $.

Finally, if $m=\inf L$, then $m\leq 0$ so $1>m$ whence there exists $r<1$ in 
$L$. From this $P$ or $\lnot P$ follows. \medskip
\end{proof}

So the notions of infimum and greatest lower bound are the same only in the
presence of the law of excluded middle. The relations that hold between the
two notions are stated in the following theorem.

\begin{theorem}
\label{Theorem glb}Let $T$ be a set of real numbers that is bounded below,
and $B$ the set of lower bounds for $T$. If $\sup B$ exists, then $\glb T=\sup B$.
If $\inf T$ exists, then $\glb T=\inf T$.
\end{theorem}

\begin{proof}
Suppose $\sup B$ exists. We must show that it is in $B$, that is, that $%
t\geq \sup B$ for all $t$ in $T$. But $t\geq b$ for each $b$ in $B$, so $%
t\geq \sup B$. (If $t<\sup B$, then there would be $b\in B$ such that $t<b$.)

Suppose $\inf T$ exists. We must show that it is the maximum element of $B$.
It is in $B$ by definition. If $b\in B$ and $b>\inf T$, then there exists $t$
in $T$ such that $t<b$, a contradiction, so $b\leq \inf T$. \medskip
\end{proof}

The usual informal description of a located subset $L$ of a metric space $M$
is that we can compute the distance from any point in $M$ to $L$. This is
slightly misleading because the rules for computing the distance are not
explicitly stated. It is not enough to be able to compute the distance from $%
x$ to $L$, we must compute the infimum%
\begin{equation*}
\inf_{y\in L}d\left( x,y\right)
\end{equation*}%
Thus there are some implied constructions of elements of $L$ hidden in that
statement about computing distances. That is not the case for the greatest
lower bound%
\begin{equation*}
\glb_{y\in L}d\left( x,y\right)
\end{equation*}%
which is also a candidate for the distance from $x$ to $L$.

To illustrate this distinction, let $M$ be the real line and $L$ the subset
of Example \ref{Example 1}. The subset $L$ contains an element, is closed,
and $\glb_{y\in L}d\left( -1,y\right) =1+\glb L=1$.
However, $\inf_{y\in L}d\left( -1,y\right) =1+\inf L$, but $L$ does not have
an infimum.

If $L$ is a subset of a metric space $M$, and $x$ is a point in $M$, then
the \textbf{distance} from $x$ to $L$ is%
\begin{equation*}
d\left( x,L\right) =\inf_{y\in L}d\left( x,y\right)
\end{equation*}%
We define the \textbf{quasidistance} from $x$ to $L$ to be 
\begin{equation*}
\delta \left( x,L\right) =\glb_{y\in L}d\left( x,y\right)
\end{equation*}%
We say that $L$ is \textbf{located} if $d\left( x,L\right) $ exists for each 
$x$ in $M$, and that $L$ is \textbf{quasilocated} if $\delta \left(
x,L\right) $ exists for each $x$ in $M$. It follows from Theorem \ref%
{Theorem glb} that a located set $L$ is quasilocated with $\delta \left(
x,L\right) =d\left( x,L\right) $ for all $x$. Note that a quasilocated set
cannot be empty.

\begin{theorem}
If $L$ is a quasilocated subset of a metric space $M$, then $\delta \left(
x,L\right) \leq \delta \left( y,L\right) +d\left( x,y\right) $ for all $x$
and $y$ in $M$. It follows that $\left\vert \delta \left( x,L\right) -\delta
\left( y,L\right) \right\vert \leq d\left( x,y\right) $ so $\delta \left(
x,L\right) $ is a uniformly continuous function of $x$.
\end{theorem}

\begin{proof}
It suffices to show that $\delta \left( x,L\right) -d\left( x,y\right) \leq
d\left( y,t\right) $ for all $t\in L$. But $d\left( x,t\right) \leq d\left(
x,y\right) +d\left( y,t\right) $, and $\delta \left( x,L\right) \leq d\left(
x,t\right) $. \medskip
\end{proof}

\noindent If we think of $\delta \left( x,L\right) $ as a lowercut (a
supremum of a bounded set of real numbers), see \cite{GR}, then the
inequality $\delta \left( x,L\right) \leq \delta \left( y,L\right) +d\left(
x,y\right) $ holds for an arbitrary subset $L$ of $M$.

The quasidistance can be thought of as dual to the distance. That's because
we can express the distance from $x$ to $L$ as%
\begin{equation*}
d\left( x,L\right) =\inf \left\{ d\left( x,t\right) :t\in L\right\} =\inf
\left\{ r:d\left( x,t\right) \leq r\text{ for some }t\text{ in }L\right\}
\end{equation*}%
while the quasidistance is%
\begin{equation*}
\delta \left( x,L\right) =\sup \left\{ r:d\left( x,t\right) \geq r\text{ for
all }t\text{ in }L\right\}
\end{equation*}

Throughout, we will use $B_{r}\left( x\right) $ to denote the open ball of
radius $r$ centered at $x$.

\begin{lemma}
\label{Lemma anchor}Let $L$ be a subset of a metric space $M$ and $x\in M$.
Then $r$ is a lower bound for $\left\{ d\left( x,y\right) :y\in L\right\} $
if and only if $B_{r}\left( x\right) \cap L=\emptyset $.
\end{lemma}

\begin{proof}
Consider the statement $P$ that there exists $y\in L$ such that $d\left(
x,y\right) <r$. Then $\lnot P$ says $d\left( x,y\right) \geq r$ for all $%
y\in L$, that is, that $r$ is a lower bound for $\left\{ d\left( x,y\right)
:y\in L\right\} $. On the other hand, $\lnot P$ says that $B_{r}\left(
x\right) \cap L=\emptyset $. \medskip
\end{proof}

\begin{theorem}
If $\delta \left( x,L\right) $ exists, then $\delta \left( x,L\right) \leq r$
if and only if $B_{s}\left( x\right) \cap L$ cannot be empty for any $s>r$.
\end{theorem}

\begin{proof}
Suppose $\delta \left( x,L\right) =\glb_{y\in L}d\left( x,y\right)
\leq r$. If $B_{s}\left( x\right) \cap L$ were empty for $s>r$, then $s$
would be a lower bound for $\left\{ d\left( x,y\right) :y\in L\right\} $, a
contradiction. Conversely, suppose $B_{s}\left( x\right) \cap L$ cannot be
empty for any $s>r$. If $\delta \left( x,L\right) >r$, then $B_{s}\left(
x\right) \cap L$ is empty for $s=\left( \delta \left( x,L\right) +r\right)
/2 $, a contradiction. \medskip
\end{proof}

We really aren't particularly interested in Example \ref{Example 1}. What
motivated these ideas was the study, in a choiceless environment, of the set 
$L$ of roots of a monic polynomial $p$ with complex coefficients. If the
coefficients of $p$ are complex numbers over the Cauchy reals, then we can
write $p\left( X\right) =\left( X-r_{1}\right) \cdots \left( X-r_{n}\right) $%
, see \cite{Wim}, and $L$ is the closure of the set $\left\{ r_{1},\ldots
,r_{n}\right\} $, so $L$ is easily seen to be located. However, if the
coefficients are arbitrary complex numbers, then we need not be able to
construct a root of $p$, see \cite{FTA}, even if $p$ is of the form $X^{2}-a$%
, so we can't show that $L$ is located. However, $L$ is always quasilocated,
as we will see (Corollary \ref{Corollary qlrootset}).

In the presence of choice, at least in complete spaces $M$ where balls are
totally bounded (such as Euclidean spaces), quasilocated zero sets are
located. To see that some hypothesis on $M$ is necessary to get this
conclusion, take $L$ as in Example \ref{Example 1}, and $M=\left\{
-1\right\} \cup L$. Then $L$ is the zero set of the function $x\left(
x-1\right) $ on $M$.

\begin{theorem}
Let $M$ be a complete space in which balls are totally bounded. Assuming the
axiom of dependent choices, quasilocated zero sets of uniformly continuous
functions on $M$ are located.
\end{theorem}

\begin{proof}
Suppose $L$ is a quasilocated subset of $M$ and $x\in M$. We will show that
if $\delta \left( x,L\right) <r$, and $\varepsilon >0$, then there exists $%
t\in L$ such that $d\left( x,t\right) \leq r+\varepsilon $. The proof is
reminiscent of the proof of the Baire category theorem. First we construct
an element $y_{1}\in B_{r}\left( x\right) $ such that $\delta \left(
y_{1},L\right) <\varepsilon /2$. This is possible because the function $%
\delta \left( z,L\right) $ is uniformly continuous, so has an infimum $m$ as 
$z$ ranges over $B_{r}\left( x\right) $. Either $m>0$ or $m<\varepsilon /2$.
But if $m>0$, then $\delta \left( x,L\right) \geq r$, so $m<\varepsilon /2$.

Similarly we construct an element $y_{2}\in B_{\varepsilon /2}\left(
y_{1}\right) $ such that $\delta \left( y_{2},L\right) <\varepsilon /4$.
Then an element $y_{3}\in B_{\varepsilon /4}\left( y_{2}\right) $ such that $%
\delta \left( y_{3},L\right) <\varepsilon /8$. In general, we construct $%
y_{n+1}\in B_{\varepsilon /2^{n}}\left( y_{n}\right) $ such that $\delta
\left( y_{n+1},L\right) <\varepsilon /2^{n+1}$. Dependent choice allows us
to form a sequence $y_{1},y_{2},\ldots $ with those properties.

Because $y_{n+1}\in B_{\varepsilon /2^{n}}\left( y_{n}\right) $, the
sequence of $y$'s is a Cauchy sequence, hence converges to a point $t$ in $M$%
. As $\delta \left( y_{n},L\right) $ converges to zero, $\delta \left(
t,L\right) =0$. As $L$ is a zero set of a uniformly continuous function, $%
t\in L$. Finally $d\left( x,y_{n}\right) <r+\varepsilon $ for all $n$, so $%
d\left( x,t\right) \leq r+\varepsilon $. \medskip
\end{proof}

With a little work, the proof of this theorem can be modified to apply to
any locally compact space $M$.

In Example \ref{Example 1}, the quasidistance from the point $0$ to the
closed set $L$ is equal to zero but $0$ need not be in $L$. That can't
happen if $L$ is the zero set of a continuous function. We write this as a
lemma for later use.

\begin{lemma}
\label{Lemma zero set}Let $f$ be a continuous function that vanishes on $L$.
If $\delta \left( x,L\right) =0$, then $f\left( x\right) =0$.
\end{lemma}

\begin{proof}
It's enough to show that $f\left( x\right) $ cannot be different from $0$.
Suppose $\left\vert f\left( x\right) \right\vert =\varepsilon >0$. As $f$ is
continuous at $x$, there is $\theta >0$ such that $f$ is nonzero in $%
B_{\theta }\left( x\right) $. Lemma \ref{Lemma anchor} then tells us that $%
\delta \left( x,L\right) \geq \theta $, a contradiction.
\end{proof}

\section{Quasifinite sets}

Even with choice, the zero set of a monic polynomial need not be finite or
even finitely enumerable. For example, let $a$ and $b$ are two real numbers
and consider the polynomial $p\left( X\right) =\left( X-a\right) \left(
X-b\right) $. In addition to the roots $a$ and $b$, the numbers $a\vee b$
and $a\wedge b$ are roots, and the ability to equate $a\vee b$ with either $%
a $ or $b$ is a form of LLPO. What we can say is that the zero set of $%
p\left( X\right) $ is the closure of the set $\left\{ a,b\right\} $. So what
we are thinking of here, in the presence of choice, are closures of finitely
enumerable sets.

What are we thinking of in the absence of choice? The example to keep in
mind is the zero set $L$ of the polynomial $X^{2}-a$ where $a$ is a small
complex number, possibly zero. Unless there is a sequence of Gaussian
numbers converging to $a$, see \cite{Wim}, we cannot necessarily construct
an element of $L$, see \cite{FTA}. But $L$ is quasilocated, as we shall see.
First a little terminology.

An $\varepsilon $\textbf{-quasiapproximation} to a subset $L$ of a metric
space $M$ is a sequence $x_{1},\ldots ,x_{n}$ in $M$ such that $%
B_{\varepsilon }\left( x_{i}\right) \cap L$ cannot be empty for $i=1,\ldots
,n$, and $L\subset B_{\varepsilon }\left( x_{1}\right) \cup \cdots \cup
B_{\varepsilon }\left( x_{n}\right) $. We say that $L$ is \textbf{%
quasi-totally-bounded} if for each $\varepsilon >0$ we can find an $%
\varepsilon $-quasi\-approximation to $L$. Unlike the notion of totally
bounded, this depends on the ambient space $M$ because we cannot assume that
the $\varepsilon $-quasi\-approximation lies in $L$. The subset $L$ in
Example \ref{Example 1} is quasi-totally-bounded but we can't find an
internal $1$-quasi\-approximation to it.

\begin{theorem}
Each quasi-totally-bounded subset of a metric space $M$ is quasi\-located.
\end{theorem}

\begin{proof}
Let $L$ be a quasi-totally-bounded subset of $M$. For any $\varepsilon >0$,
let $x_{1},\ldots ,x_{n}$ be an $\varepsilon $-quasi\-approximation to $L$.
First we show, for each $z$ in $M$, that $r=d\left( z,\left\{ x_{1},\ldots
,x_{n}\right\} \right) -\varepsilon $ is a lower bound for $T=\left\{
d\left( z,y\right) :y\in L\right\} $. Given $y$ in $L$ there is $i$ such
that $d\left( x_{i},y\right) <\varepsilon $, so $d\left( z,x_{i}\right) \leq
d\left( z,y\right) +d\left( y,x_{i}\right) <d\left( z,y\right) +\varepsilon $%
. Thus%
\begin{equation*}
r+\varepsilon \leq d\left( z,x_{i}\right) <d\left( z,y\right) +\varepsilon
\end{equation*}%
so $r<d\left( z,y\right) $. Next we show that $r+3\varepsilon $ is not a
lower bound for $T$. We know that $d\left( z,x_{i}\right) <r+2\varepsilon $
for some $i$, so $B_{\varepsilon }\left( x_{i}\right) \subset
B_{r+3\varepsilon }\left( z\right) $. We also know that $B_{\varepsilon
}\left( x_{i}\right) \cap L$ cannot be empty, so $B_{r+3\varepsilon }\left(
z\right) \cap L$ cannot be empty, whence $r+3\varepsilon $ is not a lower
bound for $T$. Thus we can approximate $\delta \left( z,L\right) $ within $%
3\varepsilon /2$. \medskip
\end{proof}

We say that $L$ is a \textbf{quasi}-$n$-\textbf{set} if for each $%
\varepsilon >0$ there is an $\varepsilon $-quasi\-approximation to $L$ of
length $n$. Note that a singleton is a quasi-$n$-set for every positive $n$,
so $n$ is just an upper bound on the size of $L$ rather than its size. We
say that $L$ is \textbf{quasifinite} if it is a quasi-$n$-set for some $n$.
Clearly a quasifinite set is quasi-totally-bounded.

The meaning of $\delta \left( x,L\right) <t$, when $\delta \left( x,L\right) 
$ does not necessarily exist, follows from the general theory of lowercuts.
It is the second condition in the following lemma.

\begin{lemma}
If $\delta \left( x,L\right) $ exists, then $\delta \left( x,L\right) <t$ if
and only if $B_{r}\left( x\right) \cap L$ cannot be empty for some $r<t$.
\end{lemma}

\begin{proof}
We know that $r\leq \delta \left( x,L\right) $ if and only if $B_{r}\left(
x\right) \cap L$ is empty (Lemma \ref{Lemma anchor}), so if $\delta \left(
x,L\right) <t$, then $B_{r}\left( x\right) \cap L$ cannot be empty for $%
r=\left( \delta \left( x,L\right) +t\right) /2$. Conversely, if $B_{r}\left(
x\right) \cap L$ cannot be empty for some $r<t$, then $r\leq \delta \left(
x,L\right) $ is impossible, so $\delta \left( x,L\right) <t$. \medskip
\end{proof}

We will show, in Corollary \ref{Corollary qlrootset}, that the zero set of a
monic polynomial is quasifinite.

\section{The spectrum of a polynomial}

Let $f$ be a monic polynomial of degree $n\geq 1$ with complex coefficients.
We can approximate $f$ arbitrarily closely with polynomials of the form 
\begin{equation*}
\left( X-q_{1}\right) \left( X-q_{2}\right) \cdots \left( X-q_{n}\right)
\end{equation*}%
where the $q_{i}$ are Gaussian numbers. The \textbf{spectrum} of $f$,
written $\spec f$, is an element of the completion of the set of $n$%
-multisets of Gaussian numbers in an appropriate metric, see \cite{FTA}.
Classically, $\spec f$ can be identified with the multiset of roots
of $f$. Constructively, with choice, $\spec f$ can be identified
with an equivalence class of multisets $r_{1},\ldots ,r_{n}$ where two
multisets are equivalent if $\prod_{i=1}^{n}\left( X-r_{i}\right)
=\prod_{i=1}^{n}\left( X-r_{i}^{\prime }\right) $.

The choiceless fundamental theorem of algebra says that the function $%
\spec$ is a homeomorphism from the space of monic polynomials of
degree $n$, with the natural metric, to the completion of the set of $n$%
-multisets of Gaussian numbers.

The \textbf{distance}, $d\left( z,\spec f\right) $, from a complex
number $z$ to $\spec f$ is approximated by the distance from $z$ to
the approximating multisets $\left\{ q_{1},q_{2},\ldots ,q_{n}\right\} $.
Note that $d\left( z,\spec f\right) $ is independent of
multiplicities. It is not hard to show that $d\left( z,\spec %
f\right) =0$ if and only if $f\left( z\right) =0$. That is the connection
between $\spec f$ and the zero set of $f$.

The \textbf{diameter} of $\spec f$ is the limit of the the diameters
of its approximating $n$-multisets. Caution: this need not be a limit of a 
\textit{sequence} of approximating $n$-multisets---we are operating in the
absence of countable choice. The diameter can be characterized in terms of $%
d\left( z,\spec f\right) $ as 
\begin{equation*}
\inf_{\varepsilon >0}\sup \left\{ \left\vert z-z^{\prime }\right\vert
:d\left( z,\spec f\right) <\varepsilon \text{ and }d\left( z^{\prime
},\spec f\right) <\varepsilon \right\}
\end{equation*}

The (set-)\textbf{distance} between $\spec f$ and $\spec g$
is defined to be the infimum over all $z$ of 
\begin{equation*}
d\left( z,\spec f\right) +d\left( z,\spec g\right) .
\end{equation*}
It is not hard to see that this infimum exists, and is equal to the limit of
the distance between the approximations to $\spec f$ and $%
\spec g$. This is like the distance between two sets of complex
numbers. It is not a metric and should not be confused with the metric
distance between the spectra of two monic polynomials of the same degree
that comes from their both being in the completion of the set of $n$%
-multisets of Gaussian numbers.

We will need this lemma.

\begin{lemma}
\label{Lemma factor}Let $f$ be a monic polynomial of degree $n\geq 1$ with
complex coefficients. If the diameter of $\spec f$ is positive, then 
$f=gh$ for nonconstant polynomials $g$ and $h$. Moreover, the distance
between $\spec g$ and $\spec h$ is positive.
\end{lemma}

\begin{proof}
Let $d$ be the diameter of $\spec f$ and let $r$ be a rational
number such that $d\geq r>0$. Let $q_{1},\ldots ,q_{n}$ be Gaussian numbers
that approximate the spectrum of $f$ to within $r/\left( 8n\right) $. The
transitive closure of the relation $\left\vert x-y\right\vert \leq r/\left(
2n\right) $ partitions the $q_{i}$ into at least two equivalence classes
(thought of as multisets). Let $G_{0}$ be one of them and $H_{0}$ the union
of all the rest. The minimum distance $\varepsilon $ between a point in $%
G_{0}$ and a point in $H_{0}$ is greater than $r/\left( 2n\right) $.

Let $m$ be the number of elements of $G_{0}$, including multiplicities. Any
multiset of Gaussian numbers that approximates the spectrum of $f$ to within 
$r/\left( 8n\right) $ is uniquely the union of an $m$-multiset $G$ and an $%
\left( m-n\right) $-multiset $H$, where $G$ is within $r/\left( 4n\right) $
of $G_{0}$, as $m$-multisets, and $H$ is within $r/\left( 4n\right) $of $%
H_{0}$. So the distance between any point in $G$ and any point in $H$ is at
least $\varepsilon -r/\left( 2n\right) $. Let $g$ and $h$ be the monic
polynomials whose spectra are the limits of the multisets $G$ and $H$
respectively. That is, $g$ is the limit of the polynomials $\Pi _{s\in
G}\left( X-s\right) $. The distance between $\spec g$ and $\spec h$ 
is at least $\varepsilon -r/\left( 2n\right) $. \medskip
\end{proof}

Any easy induction gives us the theorem.

\begin{theorem}
\label{Theorem qfspec}Let $f$ be a monic polynomial of degree $n\geq 1$ with
complex coefficients. For each $\varepsilon >0$ we can write $%
f=g_{1}g_{2}\cdots g_{k}$ where the diameter of each $\spec g_{i}$
is less than $\varepsilon $ and the $\spec g_{i}$ are a positive
distance from each other.
\end{theorem}

\begin{proof}
If the diameter $d$ of $\spec f$ is less than $\varepsilon $, then
let $k=1$ and $g_{1}=f$. Otherwise $d>0$ and from Lemma \ref{Lemma factor}
we can write $f=gh$ for nonconstant polynomials $g$ and $h$. By induction on
degree we can write $g$ and $h$ in the desired form. Clearly the spectrum of
any factor of $g$ is a positive distance from the spectrum of any factor of $%
h$. \medskip
\end{proof}

\section{Weak fundamental theorem of algebra}

We want to show that the zero set $L$ of a monic polynomial $p$ is
quasilocated. We know that the quasidistance from $z$ to $L$ has to be $%
r=d\left( z,\spec p\right) $. But to prove that, we must show that
it is impossible for $B_{r+\varepsilon }\left( z\right) \cap L$ to be empty
for $\varepsilon >0$ (Lemma \ref{Lemma anchor}). In particular, we must show
that it is impossible for the zero set of a monic polynomial of degree at
least $1$ over the complex numbers to be empty---the weak fundemental
theorem of algebra.

Suppose we have two polynomials%
\begin{eqnarray*}
a &=&a_{m}X^{m}+a_{m-1}X^{m-1}+\cdots +a_{1}X+a_{0} \\
b &=&b_{n}X^{n}+b_{n-1}X^{n-1}+\cdots +b_{1}X+b_{0}
\end{eqnarray*}%
with coefficients in a commutative ring $R$. The so-called \textbf{%
pseudodivision algorithm}, which is proved in essentially the same way as
the division algorithm for monic polynomials, says that if $0<m\leq n$, then
we can find canonical $q$ and $r$, with $\deg r<m$, such that%
\begin{equation*}
a_{m}^{n-m+1}b=qa+r
\end{equation*}%
We will denote the remainder $r$ by $\rho \left( b,a\right) $ and we say
that $\rho \left( b,a\right) $ is the result of \textbf{remaindering} $b$ by 
$a$.

The coefficients of $q$ and $r$ can be written as polynomials in the
coefficients of $a$ and $b$. The computation really takes place in the
polynomial ring%
\begin{equation*}
\mathbf{Z}\left[ \alpha _{m},\ldots ,\alpha _{0},\beta _{n},\ldots ,\beta
_{0}\right] \left[ X\right]
\end{equation*}%
where the $\alpha $'s and $\beta $'s are indeterminates, so we might as well
assume that $R$ is discrete. Note that $\rho \left( b,a\right) $ is defined
only when $0\leq \deg a<\deg b$.

Let $R$ be a discrete commutative ring. If $a$ is any nonzero polynomial in $%
R\left[ X\right] $, then $\gamma \left( a\right) $ is defined to be the
result of deleting the leading term of $a$. We say that $\gamma \left(
a\right) $ is the \textbf{chop} of $a$. Set $\gamma \left( 0\right) =0$.

If $A$ and $B$ are subsets of $R\left[ X\right] $, then $\rem\left(
A,B\right) $ is defined to be 
\begin{equation*}
\left\{ p\in R\left[ X\right] :p=\rho \left( a,b\right) \text{ or }p=\rho
\left( b,a\right) \text{ for some }\left( a,b\right) \in A\times B\right\}
\end{equation*}%
and $\chop\left( A\right) $ is defined to be $\left\{ \gamma \left(
a\right) :a\in A\right\} $. We say that $A$ is \textbf{closed under chopping
and remaindering} if $\rem\left( A,A\right) \subset A$ and $%
\chop\left( A\right) \subset A$.

\begin{lemma}
\label{Lemma chop}Given a finite set $S$ of polynomials over a discrete
commutative ring, there is a finite set of polynomials containing $S$ that
is closed under chopping and remaindering.
\end{lemma}

\begin{proof}
Define $S_{0}$ to be $S$. Let $n$ be a bound on the degrees of the
polynomials in $S_{0}$. For $i>0$, define%
\begin{equation*}
S_{i}=S_{i-1}\cup \chop\left( S_{i-1}\right) \cup \rem 
\left( S_{i-1},S_{i-1}\right)
\end{equation*}%
We will show that, for $i\geq 1$, the elements of $S_{i}\setminus S_{i-1}$
have degree at most $n-i$. That's clearly true for $i=1$, and if $i>1$, and $%
p\in S_{i}\setminus S_{i-1}$, then%
\begin{equation*}
p\in \chop\left( S_{i-1}\setminus S_{i-2}\right) \cup \rem%
\left( S_{i-1},S_{i-1}\setminus S_{i-2}\right)
\end{equation*}%
whence, by induction on $i$, we have $\deg p<n-\left( i-1\right) $ so $\deg
p\leq n-i$. In particular, the elements of $S_{n}\setminus S_{n-1}$ have
degree at most $0$, so $S_{n}$ is closed under chopping and remaindering.
\medskip
\end{proof}

Apply Lemma \ref{Lemma chop} to the set $S$ of two polynomials%
\begin{eqnarray*}
&&\alpha _{m}X^{m}+\alpha _{m-1}X^{m-1}+\cdots +\alpha _{1}X+\alpha _{0} \\
&&\beta _{n}X^{n}+\beta _{n-1}X^{n-1}+\cdots +\beta _{1}X+\beta _{0}
\end{eqnarray*}%
over the ring $\mathbf{Z}\left[ \alpha _{m},\ldots ,\alpha _{0},\beta
_{n},\ldots ,\beta _{0}\right] $, where the $\alpha $'s and $\beta $'s are
indeterminates. Let $L$ be the set of leading coefficients of polynomials in 
$S^{\prime }$. The elements of $L$ are in $\mathbf{Z}\left[ \alpha
_{m},\ldots ,\alpha _{0},\beta _{n},\ldots ,\beta _{0}\right] $. If we are
then given two polynomials%
\begin{eqnarray*}
a &=&a_{m}X^{m}+a_{m-1}X^{m-1}+\cdots +a_{1}X+a_{0} \\
b &=&b_{n}X^{n}+b_{n-1}X^{n-1}+\cdots +b_{1}X+b_{0}
\end{eqnarray*}%
with coefficients in an arbitrary commutative ring $R$, we can form the
subset $L\left( a,b\right) $ of $R$ obtained by plugging $a_{i}$ and $b_{j}$
for $\alpha _{i}$ and $\beta _{j}$ in the elements of $L$. We are now ready
to prove a general Bezout's equation for polynomials over an arbitrary
commutative ring.

\begin{theorem}
Let%
\begin{eqnarray*}
a &=&a_{m}X^{m}+a_{m-1}X^{m-1}+\cdots +a_{1}X+a_{0} \\
b &=&b_{n}X^{n}+b_{n-1}X^{n-1}+\cdots +b_{1}X+b_{0}
\end{eqnarray*}%
be polynomials over a commutative ring $R$. If each of the elements of $%
L\left( a,b\right) $ is either $0$ or invertible, then there exist $s$ and $%
t $ in $R\left[ X\right] $ such that $sa+tb$ divides both $a$ and $b$ and is
either monic or $0$.
\end{theorem}

\begin{proof}
We attempt to execute the Euclidean algorithm on $a$ and $b$. This will
succeed if, each time we create a new remainder, we can write it as a
polynomial with an invertible leading coefficient. But that is exactly what
the hypothesis of the theorem allows us to do. The elements of $L\left(
a,b\right) $ include all the coefficients of $a$ and $b$ and anything we
might get by repeatedly applying the division algorithm. Unless $a=b=0$, the
last nonzero remainder from the algorithm is of the form $sa+tb$ and divides
both $a$ and $b$. We can divide by its leading coefficient to make it monic.
\medskip
\end{proof}

As an immediate corollary we get

\begin{corollary}
\label{Corollary Bezout}Let $R$ be a ring in which any noninvertible element
is zero. If $a$ and $b$ are two polynomials in $R\left[ X\right] $ then we
can construct propositions $P_{1},\ldots ,P_{k}$ such that if the
conjunction 
\begin{equation*}
\left( P_{1}\vee \lnot P_{1}\right) \wedge \left( P_{2}\vee \lnot
P_{2}\right) \wedge \cdots \wedge \left( P_{k}\vee \lnot P_{k}\right)
\end{equation*}%
of instances of the law of excluded middle holds, then there exist $s$ and $%
t $ in $R\left[ X\right] $ such that $sa+tb$ divides both $a$ and $b$ and is
either monic or $0$. \medskip
\end{corollary}

The complex numbers are such a ring, as is any Heyting field. Why do we want
Corollary \ref{Corollary Bezout}? The relevant intuitionistic tautology is
that the negation of 
\begin{equation*}
\left( P_{1}\vee \lnot P_{1}\right) \wedge \left( P_{2}\vee \lnot
P_{2}\right) \wedge \cdots \wedge \left( P_{k}\vee \lnot P_{k}\right) \wedge
A
\end{equation*}%
implies $\lnot A$. That is easy to see by induction from the case $k=1$. For 
$k=1$, suppose $A$ and $\lnot \left( \left( P_{1}\vee \lnot P_{1}\right)
\wedge A\right) $. If $P_{1}$ is true, then we get a contradiction, so $%
P_{1} $ is false, and we get a contradiction. So if we want to prove $\lnot
A $, we are allowed to assume any finite number of instances of the law of
excluded middle. The proposition $A$ that we are interested in refuting is
that a given monic polynomial with complex coefficients has no root. That
refutation is the weak fundamental theorem of algebra.

First we prove the fundamental theorem of algebra for strongly separable
polynomials.

\begin{theorem}
\label{Theorem separable}Let $f$ be a monic polynomial with complex
coefficients. Suppose that there exist polynomials $s$ and $t$ such that $%
sf+tf^{\prime }=1$. Then $f=\left( X-r_{1}\right) \cdots \left(
X-r_{n}\right) $ where the $r_{i}$ are distinct complex numbers.
\end{theorem}

\begin{proof}
From \cite{FTA} we can find algebraic numbers $q_{1},\ldots ,q_{n}$ that
approximate the spectrum of $f$ arbitrarily closely. That means that 
\begin{equation*}
g=\left( X-q_{1}\right) \cdots \left( X-q_{n}\right)
\end{equation*}
is arbitrarily close to $f$, and that given another close approximation $%
q_{1}^{\prime },\ldots ,q_{n}^{\prime }$ to the spectrum of $f$, we can
reindex it so that $q_{i}^{\prime }$ is close to $q_{i}$ for each $i$. In
particular, the numbers $f\left( q_{i}\right) $ can be made arbitrarily
small so the numbers $f^{\prime }\left( q_{i}\right) $ are bounded away from 
$0$ because of the equation $sf+tf^{\prime }=1$. Differentiation of
polynomials is continuous, so $f^{\prime }\left( q_{1}\right) $ is close to 
\begin{equation*}
g^{\prime }\left( q_{1}\right) =\left( q_{1}-q_{2}\right) \left(
q_{1}-q_{3}\right) \cdots \left( q_{1}-q_{n}\right) .
\end{equation*}%
Hence $q_{1}$ is bounded away from $q_{2},\ldots ,q_{n}$, and similarly for
the rest of the $q$'s. Thus we can find $n$ discs, $D_{1},\ldots ,D_{n}$,
bounded away from each other, so that any sufficiently close approximation
to the spectrum of $f$ has one point in each of the discs. The points in $%
D_{i}$, as we range over approximations to the spectrum, form a coherent
system of approximations, hence determine a unique complex number $r_{i}$
because the complex numbers are complete in this choiceless sense. Clearly $%
f $ vanishes on each $r_{i}$, so $f$ can be written as desired.\medskip
\end{proof}

Now we can prove the weak fundamental theorem of algebra for arbitrary monic
polynomials.

\begin{theorem}[Weak fundamental theorem of algebra]
\label{Theorem WFTA}Let $f$ be a nonconstant monic polynomial with complex
coefficients. Then the assumption that $f$ has no roots leads to a
contradiction.
\end{theorem}

\begin{proof}
We apply Corollary \ref{Corollary Bezout} to the polynomial $f$ and its
derivative $f^{\prime }$. Assume $f$ has no roots. We will derive a
contradiction from that and the existence of polynomials $s$ and $t$ in $%
\mathbf{C}\left[ X\right] $ such that $d=sf+tf^{\prime }$ is a monic
polynomial that divides both $f$ and $f^{\prime }$. As $f^{\prime }$ is a
nonzero polynomial of degree less than $f$, the polynomial $d$ is either $1$
or a monic nonconstant proper divisor of $f$. In the latter case, we can
replace $f$ by $d$ and we are done by induction on the degree as $d$ cannot
have roots (if $\deg f=1$, then $f$ has a root). So we may assume that%
\begin{equation*}
sf+tf^{\prime }=1\text{.}
\end{equation*}%
But then Theorem \ref{Theorem separable} says that $f$ has a root. \medskip
\end{proof}

The following corollary is the point of this whole exercise.

\begin{corollary}
\label{Corollary qlrootset}Let $f$ be a nonconstant monic polynomial with
complex coefficients, $Z$ the zero set of $f$, and $z$ a complex number. If $%
d\left( z,\spec f\right) =r$, and $\varepsilon >0$, then the
assumption that $f$ has no roots in the disc of radius $r+\varepsilon $
around $z$ leads to a contradiction. Thus $d\left( z,\spec f\right)
=\delta \left( z,Z\right) $ so $Z$ is quasilocated. In fact, $Z$ is
quasifinite.
\end{corollary}

\begin{proof}
Write $f=g_{1}g_{2}\cdots g_{k}$ where the diameter of each $\spec 
g_{i}$ is less than $\varepsilon $. As $r=\inf_{i}d\left( z,\spec 
g_{i}\right) $, there is $i$ such that the diameter of $g_{i}$ plus $d\left(
z,\spec g_{i}\right) $ is less than $r+\varepsilon $. That requires
that $g_{i}$ has no roots (anywhere), which leads to a contradiction by
Theorem \ref{Theorem WFTA}.

To show that $Z$ is quasifinite, we have to find, for each $\varepsilon >0$,
an $\varepsilon $-quasi-approximation to $Z$ of length $n$. An $\varepsilon $%
-quasi-approximation to $Z$ is given by the root set $r_{1},\ldots ,r_{n}$
of a polynomial over the Gaussian numbers that is close to $f$. To show that
it is an $\varepsilon $-quasi-approximation, we need to show that $%
B_{\varepsilon }\left( r_{i}\right) \cap Z$ cannot be empty for $i=1,\ldots
,n$, and $Z\subset B_{\varepsilon }\left( r_{1}\right) \cup \cdots \cup
B_{\varepsilon }\left( r_{n}\right) $. The first condition means that $%
\delta \left( r_{i},Z\right) <\varepsilon $. But $\delta \left(
r_{i},Z\right) =d\left( r_{i},\spec f\right) $ which we can make
arbitrarily small. For the second condition, if $z\in Z$, so $f\left(
z\right) =0$, then $z$ is close to one of the $r_{i}$ because if the
polynomial $f$ is close to $\left( X-r_{1}\right) \cdots \left(
X-r_{n}\right) $, and $f\left( z\right) =0$, then $z$ is close to one of the 
$r_{i}$.\medskip 
\end{proof}

\section{Comaximal polynomials and resultants}

We would like to prove that, for monic polynomials $a$ and $b$ over $\mathbf{%
C}$, that the distance between $\spec a$ and $\spec b$ is
positive if and only if there are polynomials $s$ and $t$ such that $sa+tb=1$%
. To do this, we will show that if the distance from $\spec a$ to $%
\spec b$ is positive, then the resultant of $a$ and $b$ is different
from zero. Then we will apply a theorem valid for arbitrary commutative
rings (Theorem \ref{Theorem comaximal}) that will construct the polynomials $%
s$ and $t$.

The \textbf{resultant} of $a$ and $b$ (not necessarily monic) is the
determinant of their Sylvester matrix. Here is the \textbf{Sylvester matrix}
for $a=a_{3}X^{3}+a_{2}X^{2}+a_{1}X+a_{0}$ and $b=b_{5}X^{5}+\cdots
+b_{1}X+b_{0}$%
\begin{equation*}
S=\left( 
\begin{array}{cccccccc}
a_{3} & a_{2} & a_{1} & a_{0} & 0 & 0 & 0 & 0 \\ 
0 & a_{3} & a_{2} & a_{1} & a_{0} & 0 & 0 & 0 \\ 
0 & 0 & a_{3} & a_{2} & a_{1} & a_{0} & 0 & 0 \\ 
0 & 0 & 0 & a_{3} & a_{2} & a_{1} & a_{0} & 0 \\ 
0 & 0 & 0 & 0 & a_{3} & a_{2} & a_{1} & a_{0} \\ 
b_{5} & b_{4} & b_{3} & b_{2} & b_{1} & b_{0} & 0 & 0 \\ 
0 & b_{5} & b_{4} & b_{3} & b_{2} & b_{1} & b_{0} & 0 \\ 
0 & 0 & b_{5} & b_{4} & b_{3} & b_{2} & b_{1} & b_{0}%
\end{array}%
\right)
\end{equation*}%
If $a=\prod_{i=1}^{m}\left( X-q_{i}\right) $ and $b=\prod_{j=1}^{n}\left(
X-r_{j}\right) $, then the determinant of $S$ is $\prod_{i,j}\left(
q_{i}-r_{j}\right) $. That fact is a simply a polynomial identity in the
indeterminants $q_{i}$ and $r_{j}$.

We say that two elements $a$ and $b$ a ring are \textbf{comaximal} if there
exist $s$ and $t$ in the ring such that $sa+tb=1$.

\begin{theorem}
\label{Theorem comaximal}Let $a=a_{m}X^{m}+a_{m-1}X^{m-1}+\cdots+a_{0}$ and $%
b=b_{n}X^{n}+b_{n-1}X^{n-1}+\cdots+a_{0}$ be polynomials in $R\left[ X\right]
$ with $m,n\geq1$. Then the following two conditions are equivalent:

\begin{enumerate}
\item The polynomials $a$ and $b$ are comaximal in $R\left[ X\right] $ and
the (formal) leading coefficients $a_{m}$ and $b_{n}$ are comaximal in $R$,

\item The resultant of $a$ and $b$ is a unit in $R$.
\end{enumerate}
\end{theorem}

\begin{proof}
If $S$ is the Sylvester matrix of $a$ and $b$, and $S^{\ast }$ is the
adjugate of $S$, then $S^{\ast }S=\left( \det S\right) I$, where $I$ is the
identity matrix and $\det S$ is the resultant of $a$ and $b$. So the last
row of $S^{\ast }$ gives the coefficients of polynomials $s$ and $t$ such
that $sa+tb$ is the resultant of $a$ and $b$. Thus if condition 2 holds,
then the polynomials $a$ and $b$ are comaximal in $R\left[ X\right] $. It
remains to show that $a_{m}$ and $b_{n}$ are comaximal in $R$. But that
follows by expanding the determinant of $S$ by minors using the first column.

Now suppose condition 1 holds, so there exist polynomials $s$ and $t$ such
that $sa+tb=1$. We will first prove 2 under the assumption that $a_{m}$ (or $%
b_{n}$) is a unit. Now if $g$ is any polynomial of degree less than $m+n$,
then we can certainly find polynomials $s$ and $t$ such that $sa+tb=g$. As $%
a_{m}$ is unit, we can write $t=qa+r$ where $\deg r<m$. So%
\begin{equation*}
\left( s+qb\right) a+rb=g
\end{equation*}
which means we may assume that $\deg t<m$, so $\deg tb<m+n$. This forces $%
\deg s<n$ because if $s_{k}$ is the (formal) leading coefficient of $s$, and 
$k\geq n$, then the equation $sa+tb=g$ requires $s_{k}=0$ as $\deg tb<m+n$.
Because $\deg s<n$ and $\deg t<m$, we can put the coefficients of $s$ and $t$
together to form a row vector $v$ of length $m+n$ so that $vS=\left(
g_{m+n-1},g_{m+n-2},\ldots,g_{1},g_{0}\right) $ where $S$ is the Sylvester
matrix. Since we can do this for arbitrary $g$ of degree less than $m+n$,
this shows that the matrix $S$ is invertible, whence $\det S$, which is the
resultant of $a$ and $b$, is a unit in $R$.

To finish the proof, we pass to the rings $R\left[ 1/a_{m}\right] $ and $R%
\left[ 1/b_{n}\right] $. The kernel of the natural map $R\rightarrow R\left[
1/a_{m}\right] $ is the set of elements of $R$ that annihilate some power of 
$a_{m}$. In particular, $R\left[ 1/a_{m}\right] $ is trivial exactly when $%
a_{m}$ is nilpotent. Now since the theorem is true when either $a_{m}$ or $%
b_{n}$ is a unit, it follows that $\det S$ is a unit in $R\left[ 1/a_{m}%
\right] $ and in $R\left[ 1/b_{n}\right] $. We want to show that $\det S$ is
a unit in $R$. That follows from the fact that $a_{m}$ and $b_{n}$ are
comaximal. Specifically, we have $\left( \det S\right) r/a_{m}^{k}=1$ in $R%
\left[ 1/a_{m}\right] $ for some $r\in R$ and positive integer $k$, so $%
\left( \det S\right) ra_{m}^{i}=a_{m}^{k+i}$ in $R$ for some positive
integer $i$. Thus $\det S$ divides a power of $a_{m}$ in $R$. Similarly, $%
\det S$ divides a power of $b_{n}$. That makes $\det S$ a unit because any
power of $a_{m}$ is comaximal with any power of $b_{n}$. \medskip
\end{proof}

It's not enough in Theorem \ref{Theorem comaximal} just to require that $a$
and $b$ be comaximal. A trivial example is $a=b=0X+1$ where $m=n=1$.

Over a discrete field, Theorem \ref{Theorem comaximal} is often stated as
the resultant of $a$ and $b$ is zero if and only if $a$ and $b$ have a
nontrivial common factor, it being assumed that nobody is so silly as to
have the formal leading coefficient of $a$ and $b$ equal to zero. The
theorem in this form, even for monic $a$ and $b$, does not generalize well
to arbitrary commutative rings.

Consider the polynomials $\left( X-2\right) \left( X-6\right) =X^{2}+4$ and $%
X\left( X-4\right) $ over $\mathbf{Z}_{8}$. It is easily checked that their
resultant is zero. However, they have no nontrivial common factors. It's
easy to see that they have no nontrivial \emph{monic} common factors. What
constitutes a trivial common factor? It seems clear that these are exactly
the unit polynomials, the polynomials with inverses. Over a decent ring, the
units in $R\left[ X\right] $ are just the units of $R$, but in a ring with
nilpotent elements, like $\mathbf{Z}_{8}$, there are units of arbitrarily
high degree, like $4X^{n}+1$.

We will show that over $\mathbf{Z}_{8}$, or any primary ring, the factors of
a monic polynomial are all products of a unit polynomial and a monic
polynomial. Thus if a monic polynomial has a nontrivial factor, then it has
a nontrivial monic factor.

First we note that in any ring, the factorization of a polynomial as a monic
polynomial times a unit polynomial, if possible, is unique. Indeed, if $%
mu=m^{\prime }u^{\prime }$, then $m=m^{\prime }u^{\prime }u^{-1}$ so $\deg
m^{\prime }\leq \deg m$, hence, by symmetry $\deg m^{\prime }=\deg m$. As
both $m$ and $m^{\prime }$ monic, we must have $u^{\prime }u^{-1}=1$.

We can characterize the polynomials that can be factored in this way.

\begin{theorem}
\label{Theorem unitmonic}Let $p=\sum a_{i}X^{i}$ be a polynomial over a
commutative ring $R$. Suppose $a_{m}$ is a unit in $R$ and that $a_{i}$ is
nilpotent for $i>m$. Then $p$ can be written as a unit polynomial times a
monic polynomial.
\end{theorem}

\begin{proof}
Let $I$ be a fixed nilpotent ideal of $R$ such that $a_{i}\in I$ for all $%
i>m $. We will multiply $p$ by a sequence of unit polynomials until it
becomes monic.

Suppose $j>m$ and $d\geq1$ are such that%
\begin{align*}
a_{i} & \in I^{d}\text{ for }m<i\leq j \\
a_{i} & \in I^{d+1}\text{ for }i>j
\end{align*}
Multiply $p$ by the unit polynomial%
\begin{equation*}
1-\frac{a_{j}}{a_{m}}X^{j-m}
\end{equation*}
In the resulting polynomial $p^{\prime}$ we have%
\begin{align*}
a_{m}^{\prime} & \in a_{m}+I^{d}\text{ is a unit} \\
a_{j}^{\prime} & =0\in I^{d+1} \\
a_{i}^{\prime} & \in I^{d}+I^{d}=I^{d}\text{ for }m<i\leq j \\
a_{i}^{\prime} & \in I^{d+1}+I^{d+d}=I^{d+1}\text{ for }i>j
\end{align*}
When $j=m+1$, we get $a_{i}^{\prime}\in I^{d+1}$ for all $i>m$. At that
point we set $d$ equal $d+1$ and continue until $I^{d}=0$. At that point the
polynomial has a unit $a_{m}$ as its formal leading coefficient. \medskip
\end{proof}

The converse of Theorem \ref{Theorem unitmonic} is also true because the
units in the ring $R\left[ X\right] $ are exactly those polynomials whose
constant term is a unit in $R$ and whose other coefficients are nilpotent.

We end this section by applying Theorem \ref{Theorem comaximal} to
polynomials over the complex numbers and their spectra.

\begin{theorem}
If $a$ and $b$ are nonconstant monic polynomials over $\mathbf{C}$, then the
distance between $\spec a$ and $\spec b$ is positive if and
only if $a$ and $b$ are comaximal.
\end{theorem}

\begin{proof}
We will apply Theorem \ref{Theorem comaximal}. As $a$ and $b$ are monic,
their leading coefficients are certainly comaximal. So $a$ and $b$ are
comaximal if and only if their resultant is a unit, that is, if and only if
the resultant is different from zero. It thus suffices to show that the
resultant is different from zero if and only if the distance between $%
\spec a$ and $\spec b$ is positive.

The resultant, as a function of the coefficients of $a$ and $b$, is clearly
uniformly continuous on bounded subsets. So if $\prod_{i=1}^{m}\left(
X-q_{i}\right) $ and $\prod_{j=1}^{n}\left( X-r_{j}\right) $ are close to $a$
and $b$ respectively, then $\prod_{i,j}\left( q_{i}-r_{j}\right) $ is close
to the resultant of $a$ and $b$. Thus the resultant of $a$ and $b$ is the
limit of $\prod_{i,j}\left( q_{i}-r_{j}\right) $. On the other hand, the
distance between $\spec a$ and $\spec b$ is the limit of $%
\min \left\vert q_{i}-r_{j}\right\vert $. Clearly $\prod_{i,j}\left(
q_{i}-r_{j}\right) $ is eventually bounded away from zero if and only if $%
\min \left\vert q_{i}-r_{j}\right\vert $ is. \medskip
\end{proof}

\section{Examples of Riesz spaces}

The question is raised in \cite{CS} as to whether you can construct the
points of the spectrum of a separable Riesz space, all of whose elements are
normable, in the absence of countable choice. The authors, referring to \cite%
{FTA}, suggested basing a counterexample on the zero set of $X^{2}-a$.
However it appeared that the resulting Riesz space was not separable. Here
we elaborate on that construction and give a separable counterexample.

A \textbf{Riesz space} is a lattice ordered vector space $V$ over the
rational numbers with an element $1$ such that for each $v$ in $V$, there
exists a positive integer $n$ such that $v\leq n\cdot 1$. The canonical
example is the space of uniformly continuous functions on a compact metric
space, where $1$ is the constant function whose value is $1$. We say that $V$
is \textbf{normable}, or \textbf{normed}, if for each $v\in V$, the set $%
\left\{ q\in \mathbf{Q}:v\leq q\cdot 1\right\} $ has an infimum in $\mathbf{R%
}$. We call this greatest lower bound the \textbf{least upper bound} of $v$.
The norm of $f$ is defined to be the least upper bound of $\left\vert
v\right\vert =v\vee 0-v\wedge 0$. Note that classically every Riesz space is
normed.

Let $S$ be a quasilocated subset of the complex plane that is contained in a
closed disc $D$. Let $C\left( D\right) $ be the Riesz space of uniformly
continuous real valued functions on $D$. Let $K$ be the functions in $%
C\left( D\right) $ that vanish on $S$. The set $K$ is a subspace of $C\left(
D\right) $ that is closed under the lattice operations. Our counterexample
is the space $C\left( D\right) /K$. It is separable because $C\left(
D\right) $ is. Note that this is not exactly the space of uniformly
continuous functions on $S$, it is the space of those uniformly continuous
functions on $S$ that extend to uniformly continuous functions on $D$.

If $f\in C\left( D\right) $, then $f$ vanishes on $S$ if and only if $f$ and 
$-f$ are nonnegative on $S$. Note that if $f$ is nonnegative on $S$, and $g$
vanishes on $S$, then $f+g$ is nonnegative on $S$, so this defines a notion
of nonnegativity on $C\left( D\right) /K$. You get another notion of
nonnegativity on $C\left( D\right) /K$ by considering those functions that
come from nonnegative functions in $C\left( D\right) $. However, if $f$ is
nonnegative on $S$, then $f$ is equal to $f\vee 0$ on $S$, so these two
notions are the same.

\begin{theorem}
\label{Theorem Riesz}Let $S$ be a quasilocated subset of the complex plane
that is contained in a closed disc $D$. Let $C\left( D\right) $ be the Riesz
space of uniformly continuous real valued functions on $D$ and let $K$ be
the functions in $C\left( D\right) $ that vanish on $S$. Then every element
of $C\left( D\right) /K$ is normable, and each Riesz-space homomorphism $%
C\left( D\right) /K\rightarrow R$ that takes $1$ to $1$ is given by
evaluation at a point $z_{0}\in D$ such that $\delta \left( z_{0},S\right)
=0 $. Conversely, if $\delta \left( z_{0},S\right) =0$, and $f\in K$, then $%
f\left( z_{0}\right) =0$, so evaluation at $z_{0}$ gives a Riesz-space
homomorphism $C\left( D\right) /K\rightarrow R$ that takes $1$ to $1$.
\end{theorem}

\begin{proof}
We want to approximate the least upper bound of $f$ on $S$ to within $%
\varepsilon $. Choose $\theta >0$ so that if $d\left( x,y\right) \leq \theta 
$, then $\left\vert f\left( x\right) -f\left( y\right) \right\vert
<\varepsilon /2$. Take a finite $\theta /4$-approximation $F$ to $D$ and
partition $F$ into two subsets, one, $S^{\prime }$, where $\delta \left(
x,S\right) <\theta $ and one where $\delta \left( x,S\right) >\theta /2$.
Let $\mu $ be the supremum of $f$ on $S^{\prime }$. We want to show that $%
f\leq \mu +\varepsilon $ on $S$ and that it is not the case that $f\leq \mu
-\varepsilon $ on $S$.

To show that $f\leq \mu +\varepsilon $ on $S$, let $s\in S$. There is $%
s^{\prime }\in S^{\prime }$ such that $d\left( s,s^{\prime }\right) <\theta $
for otherwise $d\left( s,s^{\prime }\right) >\theta /2$ for every $s^{\prime
}\in S^{\prime }$. But $d\left( s,t\right) <\theta /2$ for some $t\in F$,
and this $t$ must be in $S^{\prime }$. Thus $\left\vert f\left( s\right)
-f\left( s^{\prime }\right) \right\vert <\varepsilon /2$. As $f\left(
s^{\prime }\right) \leq \mu $, we conclude that $f\left( s\right) \leq \mu
+\varepsilon /2$.

Now suppose $f\leq \mu -\varepsilon $ on $S$. We know that there is $%
s^{\prime }\in S^{\prime }$ such that $f\left( s^{\prime }\right) $ is
arbitrarily close to $\mu $. Note that $B_{\theta }\left( s^{\prime }\right)
\cap S$ cannot be empty because $\delta \left( s^{\prime },S\right) <\theta $%
. Yet if $s\in B_{\theta }\left( s^{\prime }\right) \cap S$, then $%
\left\vert f\left( s\right) -f\left( s^{\prime }\right) \right\vert
<\varepsilon /2$, so $f\left( s^{\prime }\right) \leq \mu -\varepsilon /2$,
a contradiction.

For the second claim, let $\rho $ be a Riesz-space homomorphism $C\left(
D\right) /K\rightarrow \mathbf{R}$ taking $1$ to $1$. This induces a
homomorphism $\rho :C\left( D\right) \rightarrow \mathbf{R}$ taking $1$ to $1
$ and $K$ to zero. To see that $\rho $ is given by evaluation at a point in $%
D$, let $W$ be the vector sublattice of $C\left( D\right) $ generated by $1$
and the two coordinate projections, $\pi _{1}$ and $\pi _{2}$. Then $W$ is
dense in $C\left( D\right) $ by the Stone-Weierstrass approximation theorem
(see below). Let $z_{0}=\rho \left( \pi _{1}\right) +i\rho \left( \pi
_{2}\right) $. It's easy to see that $\rho $ is evaluation at $z_{0}$ on $W$%
, hence is evaluation at $z_{0}$ on $C\left( D\right) $.

We want to show that $\delta \left( z_{0},S\right) =0$. The function $%
f\left( x\right) =\delta \left( x,S\right) $ is in $C\left( D\right) $ and
vanishes on $S$. So $\rho \left( f\right) =0$. But $\rho $ is evaluation at $%
z_{0}$, so $\delta \left( z_{0},S\right) =0$. Conversely, suppose $\delta
\left( z_{0},S\right) =0$ and $f\in K$. Then Lemma \ref{Lemma zero set}
tells us that $f\left( z_{0}\right) =0$. \medskip
\end{proof}

\begin{corollary}
Let $S$ be the root set of a nonconstant monic polynomial $p$ and let $%
V=C\left( D\right) /K$ be the Riesz space associated with $S$ as in the
theorem. Then $V$ is normed and separable, and each Riesz-space homomorphism
from $V$ onto $\mathbf{R}$ gives a root of $p$.
\end{corollary}

\begin{proof}
The space $V$ is separable because the space $C\left( D\right) $ is
separable. As $p$ is in $K$, if $\delta \left( z_{0},S\right) =0$, then $%
z_{0}$ is a root of $p$. \medskip
\end{proof}

The point of the corollary is that, in the absence of choice, we need not be
able to construct a root of the polynomial $p\left( X\right) =X^{2}-a$, see 
\cite{FTA}, so we need not be able to construct a Riesz-space homomorphism
from a normed, separable, Riesz space onto $\mathbf{R}$.

We include a short choice-free proof of the Stone-Weierstrass theorem for
the square, which was appealed to in the proof of Theorem \ref{Theorem Riesz}%
. We want to approximate nonegative functions on the square by finite
suprema of finite infima of functions of the form $ax+by+c$, that is,
functions whose graphs are planes. Suppose we have a little square $S_{i}$
in a grid and a nonnegative constant $c_{i}$. Let $\varphi _{i}$ be the
infimum of the constant $c_{i}$ and one affine function for each side $L$ of 
$S_{i}$ whose value is $c_{i}$ on $L$ and $0$ on the opposite side of the
other little square in the grid with side $L$. The graph of $\varphi _{i}$
is a pyramid with its top cut off. Let $\psi =\sup_{i}\varphi _{i}$. Let $J$
be the indexes of the block of nine little squares centered at $S_{i}$. If $%
x\in S_{i}$, then 
\begin{equation*}
\inf_{j\in J}c_{j}\leq \psi \left( x\right) \leq \sup_{j\in J}c_{j}
\end{equation*}%
Given a uniformly continuous function $f$ on the square, choose a grid so
that $f$ does not vary by more than $\varepsilon $ on any block of nine
little squares in the grid. Choose $c_{i}$ to be the value of $f$ in the
center of the little square $S_{i}$ and let $\psi $ be as above. Then $\psi $
approximates $f$ within $\varepsilon $.


\bibliographystyle{amsplain}

\end{document}